\documentclass[a4paper, 12pt, leqno]{amsart}

\usepackage{amsmath,amsfonts}
\usepackage{lipsum}
\usepackage[active]{srcltx}
\usepackage{pslatex}
\usepackage{amsmath}
\usepackage{graphicx}
\usepackage{rotating}
\usepackage{caption}
\usepackage{fancyhdr}
\usepackage{hyperref}
\usepackage{amsaddr}
\usepackage{esint}
\usepackage{dsfont}
\usepackage{color}
\begin{document}
\pagestyle{plain}
\bibliographystyle{plain}
\newtheorem{theo}{Theorem}[section]
\newtheorem{lemme}[theo]{Lemma}
\newtheorem{cor}[theo]{Corollary}
\newtheorem{defi}[theo]{Definition}
\newtheorem{prop}[theo]{Proposition}
\newtheorem{problem}[theo]{Problem}
\newtheorem{remarque}[theo]{Remark}
\newcommand{\beq}{\begin{eqnarray}}
\newcommand{\enq}{\end{eqnarray}}
\newcommand{\be}{\begin{eqnarray*}}
\newcommand{\en}{\end{eqnarray*}}
\newcommand{\ben}{\begin{eqnarray*}}
\newcommand{\enn}{\end{eqnarray*}}
\newcommand{\Td}{\mathbb T^d}
\newcommand{\Rd}{\mathbb R^n}
\newcommand{\R}{\mathbb R}
\newcommand{\N}{\mathbb N}
\newcommand{\Sn}{\mathbb S}
\newcommand{\Zd}{\mathbb Z^d}
\newcommand{\Linf}{L^{\infty}}
\newcommand{\dt}{\partial_t}
\newcommand{\Dt}{\frac{d}{dt}}
\newcommand{\Dtt}{\frac{d^2}{dt^2}}
\newcommand{\demi}{\frac{1}{2}}
\newcommand{\vf}{\varphi}
\newcommand{\epu}{_{\varepsilon}}
\newcommand{\ep}{^{\varepsilon}}
\newcommand{\bfi}{{\mathbf \Phi}}
\newcommand{\bpsi}{{\mathbf \Psi}}
\newcommand{\bx}{{\mathbf x}}
\newcommand{\dis}{\displaystyle}
\newcommand{\ds}{\partial_s}
\newcommand{\dss}{\partial_{ss}}
\newcommand{\dx}{\partial_x}
\newcommand{\dxx}{\partial_{xx}}
\newcommand{\dy}{\partial_y}
\newcommand{\dyy}{\partial_{yy}}
\newcommand{\dtt}{\partial_{tt}}
\newcommand {\g}{\`}
\newcommand{\E}{\mathbb E}
\newcommand{\bQ}{\mathbb Q}
\newcommand{\1}{\mathbb I}
\newcommand{\bF}{\mathbb F}
\newcommand{\F}{\cal F}
\newcommand{\bP}{\mathbb P}
\let\cal=\mathcal
\newcommand{\lb}{\langle}
\newcommand{\rb}{\rangle}
\newcommand{\bu}{\bar{u}}
\newcommand{\uu}{\underline{u}}
\newcommand{\bv}{\bar{v}}
\newcommand{\uv}{\underline{v}}
\newcommand{\cS}{{\cal S}}
\newcommand{\cSf}{\cS_\Phi}
\newcommand{\bw}{\bar{w}}
\newcommand{\uw}{\underline{w}}
\newcommand{\bsig}{\bar{\sigma}}
\newcommand{\usig}{\underline{\sigma}}
\renewcommand{\theequation}{\arabic{section}.\arabic{equation}}
\def\red#1{\textcolor[rgb]{1,0,0}{ #1}}

\author{Gr\'egoire Loeper$^{1}$, Fernando Quir\'os$^2$} 
\address{$^{1}$Monash University, School of Mathematical Sciences\\
$^2$Universidad Aut\'onoma de Madrid}
\title{Interior second derivatives estimates for nonlinear diffusions}
\date{\today}
 

\begin{abstract}
By an extension of of some estimates due to Crandall and Pierre \cite{CranPieA} and Di Benedetto \cite{DiBenedettoDegen} we derive consequences for fully nonlinear parabolic equations of the form
$\dt v + F(t,x,D^2v)=0$,
where $F$ can be both singular and degenerate elliptic and also non-homogeneous. Such equations appear in the theory of option pricing with market impact. 
\end{abstract}

\maketitle
\section{Introduction}
\label{sect-introduction} \setcounter{equation}{0}

The original motivation for this paper is the study of fully nonlinear parabolic partial differential equations of the form
\beq\label{eq:hjb-intro}
\dt v + F(t,x,\dxx v)=0,
\enq
where $u$ is defined in $[0,T]\times \R$, the terminal condition  $u(T,\cdot)$ is given, and the solution is solved backwards in time. We investigate the case where $F(t,x,\gamma)$ is typically a convex function in its third argument, with its derivative $F_\gamma$ going from $0$ at $-\infty$ to $+\infty$ at $\bar\gamma$ (potentially $\bar\gamma=\infty$). One  example is 
\beq\label{eq.aberloep}
\dt v +\demi \sigma^2(t,x)\left(a  +\frac{b}{(1-\lambda \dxx v)^{p_1}} + \frac{c}{(1-\lambda \dxx v)^{p_2}}\right)=0,
\enq
for $t\in [0,T], x\in \R$,
which comes from theory of option pricing with {\it market impact}, see \cite{AbergelLoeper,  BoLoZo1, BoLoZo2, LoMi1, BoLoSonZou}. There, $0<p_1<p_2$, $\lambda>0$, and $\sigma$ is a bounded Lipschitz function such that $\inf\sigma>0$.  
The conditions  $b, c >0$ guarantee that the equation is parabolic as long as $\lambda \dxx v < 1$, and $a+b+c=0$ ensures that constants are solutions.

The equation is singular when $\lambda \dxx v \to 1^-$ and degenerate when $\dxx v \to -\infty$. Our aim is to obtain a priori interior estimates for the second derivatives guaranteeing that the equation is neither degenerate nor singular if we are away from the terminal time $T$. Namely, we will prove that, if there exists a supersolution, then, for any $\tau>0$, there exists  some $\varepsilon(\tau)>0$ such that 
\[-\varepsilon^{-1} \leq \dxx v\leq \lambda^{-1} - \varepsilon \quad \text{for }t\le T-\tau.
\]
Consequently the equation is uniformly parabolic away from the terminal time and higher regularity follows by standard arguments. 

General equations of the form \eqref{eq:hjb-intro} with singular behaviour are also met in some problems related to optimal transport by diffusions, see \cite{TanTouz, GuoLeoLo, GuoLo}.

Some of our results are quite general and apply to solutions of
\beq\label{eq:intro-v}
\dt v = F(t,x,-A(v)),
\enq
for $A$ an {\it accretive} operator as in ~\cite{CranPieA}.
The most important cases will be $A=-\dxx$, or $A=-\Delta$ in higher dimensions. To obtain our results, we will study the equation followed by $u=-Av$:
\begin{equation}
\label{eq:non.separable}
\partial_t u + A(F(t,x,u))=0.
\end{equation}

Our paper consists of three estimates for solutions to~\eqref{eq:non.separable} which have independent interest.

The first result is a generalisation of the classical estimate obtained by Aronson and B\'enilan in~\cite{Aronson-Benilan} for the time derivative of non-negative solutions of~\eqref{eq:non.separable} when $A=-\Delta$ and $F(t,x,u)=u^m$, $m>(d-2)^+/d$, where $d$ is the spatial dimension. This estimate was later extended by Crandall and Pierre to the case in which $F(t,x,u)=\varphi(u)$, under some assumptions on $\varphi$, first for $A=-\Delta$ in~\cite{CranPieDelta}, and later for general accretive operators in~\cite{CranPieA}. Here we generalize this last result to the case in which $F$ is not homogeneous, neither in space nor in time,  giving an unconditional (i.e. independent of the initial data) information on $\dt u$. It is somewhat a surprise that there is no need for any regularity of $F$ with respect to $x$, only with respect to $t$ and $u$. These results are given first in the separable case, $F(x,t,u)=\kappa(t,x)\varphi(u)$, in Theorem~\ref{theo-u}, and are later extended to the general non-separable case in Theorem~\ref{theo:non.separable}. 

The second result, Theorem~\ref{theo-v},  is a consequence of Theorem~\ref{theo:non.separable} for solutions to~\eqref{eq:intro-v}
when $F$ can be singular for large values of $-A(v)$, still under some structure condition on the behavior of $F$ with respect to $u$. We show interior $C^2$ regularity under the assumption of the existence of a supersolution.

The third result, Theorem \ref{theo:degen},  shows 
expansion of positivity  for equations of the form 
\ben
\dt v = F(t,D_x v, \Delta v),
\enn
with $v$ convex, and $F(t,p,z)$  singular for $z\sim0$. 
This result is in the spirit of the one of Di Benedetto  \cite{DiBenedettoDegen}, in a case where we have gradient dependency. Under a Legendre transform, this result will imply the bound from below for $\dxx v$ in equation \eqref{eq:hjb-intro}.

Building on these results we deduce the interior regularity for solutions of~\eqref{eq:hjb-intro} in Theorem~\ref{theo:reg-aber-loep}.

\section{Time derivative estimate and applications to the singular case}
\label{sect-time.derivative.estimate} \setcounter{equation}{0}

In this section we generalize  the time derivative estimate obtain by B\'enilan and Crandall in~\cite{CranPieA} and derive consequences for singular partial differential equations that appear in option pricing. 
\subsection{The operator} 
As in~\cite{CranPieA}, we assume that:
\begin{itemize} 
	\item $A$ is a densily defined,  $m$-accretive in $L^1(\mathbb{R}^d)$ linear operator.   
    \item	If $u\in L^1(\mathbb{R}^d)\cap \Linf(\mathbb{R}^d)$ and $\beta$ is a monotone graph in $\R\times \R$ with $0\in \beta(0), v\in \beta(u)$ then 
	\begin{equation}
	\label{A3}\int vA(u) \, dx \geq 0.
	\end{equation}
\end{itemize}

Thanks to~\eqref{A3} we have a comparison principle, which will be important in the sequel. 
\begin{lemme}
	The comparison principle holds for solutions in $L^1 \cap L^\infty$ of equation~\eqref{eq:non.separable}. 
\end{lemme}
\begin{proof} Assume that $u(0) \geq v(0)$, take the difference of the equations \eqref{eq:non.separable} for $u$ and $v$,  multiply by $\mathds{1}_{u \leq v}$, and use~\eqref{A3} to conclude.
\end{proof}

\subsection{The separable case}
Let $u$ be a non-negative solution on $t>0$ to 
\beq\label{main}
\dt u  + A(\kappa(t,x) \vf(u))=0.
\enq
Under an structural assumption on $\varphi$, which coincides with that in~\cite{CranPieA} for the case in which $\kappa=1$, and with some regularity hypothesis on $\kappa$,   there is an unconditional estimate for the time derivative of non-negative solutions of~\eqref{main}, as we show next.
\begin{theo}\label{theo-u}
Let $u$ be a non-negative classical solution to \eqref{main} on $[0,T]$ belonging to $L^1(\mathbb{R}^d)\cap \Linf(\mathbb{R}^d)$, and assume that $\vf$ is non-decreasing, with $\vf(0)=0$ and satisfies for some $m>0$, $\theta\in \{-1,1\}$
\beq\label{m}
\inf_{u\geq 0}\left\{ \theta \frac {\vf(u)\vf''(u)} {(\vf'(u))^2}\right\} \geq m.
\enq
Assume also that $\kappa$ is positive and such that
\beq\label{eq:conditions.kappa}
\sup_{t\in[0,T],\,x}\{ \kappa, \kappa^{-1}, |\dt \kappa|, |\dtt \kappa|\} \leq L
\enq
for some constant $L>0$.
Then there exists a constant $\rho>0$ depending only on $m, L, T$ such that 
\beq
\label{eq:monotonicity.separable}
t \to \theta t^{\rho\theta}\kappa(t,x)\vf(u(t,x)) \quad \text{is non-decreasing in }[0,T].
\enq
\end{theo}

\begin{proof} We consider
\beq\label{defv}
w=t\dt u + \left(\theta\rho + t\frac{\dt \kappa}{\kappa}\right) \frac{\vf(u)}{\vf'(u)},
\enq
where $\rho>0$ is a constant  to be chosen later.
Differentiating equation \eqref{main} with respect to time we get
\ben
\dtt u  + A\left(\kappa\vf'(u) \left( \dt u +  \frac{\dt \kappa}{\kappa}\frac{\vf(u)}{\vf'(u)}\right)\right)=0,
\enn
wich reads also
\ben
\dtt u + \frac{1}{t}A\left(\kappa \vf'(u) \left(w - \theta\rho\frac{\vf(u)}{\vf'(u)}\right)\right)=0,
\enn
while differentiating \eqref{defv} we obtain
\ben
\dt w = t\dtt u + \dt u + \left(\theta\rho + t\frac{\dt \kappa}{\kappa}\right) \dt u \left(1-\frac{\vf(u) \vf''(u)}{\vf'^2(u)}\right) +\dt\left( t \frac{\dt \kappa}{\kappa}\right) \frac{\vf(u)}{\vf'(u)}.
\enn
Combining these two identities with \eqref{main} and \eqref{defv} we obtain
\[
\begin{array}{l}
\displaystyle
\dt w + A\left(\kappa \vf'(u) w\right)=\theta\rho A\left(\kappa \vf(u)\right) 
+\dt u \\[10pt]
\displaystyle\qquad+ \left(\theta\rho + t\frac{\dt \kappa}{\kappa}\right) \dt u \left(1-\frac{\vf(u) \vf''(u)}{\vf'^2(u)}\right) +\dt\left( t \frac{\dt \kappa}{\kappa}\right) \frac{\vf(u)}{\vf'(u)}\\[10pt]
\displaystyle\quad= \dt u \left(1+t\frac{\dt \kappa}{\kappa} - \left(\theta\rho + t\frac{\dt \kappa}{\kappa}\right)\frac{\vf(u) \vf''(u)}{\vf'^2(u)}\right)
+\dt\left( t \frac{\dt \kappa}{\kappa}\right) \frac{\vf(u)}{\vf'(u)}\\[10pt]
\displaystyle\quad= \frac{1}{t}\left(w - \left(\theta\rho + t \frac{\dt\kappa}{\kappa}\right)\frac{\vf(u)}{\vf'(u)}\right)\left(1+t\frac{\dt \kappa}{\kappa} - \left(\theta\rho + t\frac{\dt \kappa}{\kappa}\right)\frac{\vf(u) \vf''(u)}{\vf'^2(u)}\right)\\[10pt]
\displaystyle\qquad
+\dt\left( t \frac{\dt \kappa}{\kappa}\right) \frac{\vf(u)}{\vf'(u)}.
\end{array}
\]
Defining 
$$
\tilde{\rho}=\rho  + \theta\, t \frac{\dt\kappa}{\kappa},\qquad
Q=-\left(1+t\frac{\dt \kappa}{\kappa}\right) + \tilde\rho \theta \frac{\vf(u) \vf''(u)}{\vf'^2(u)},
$$
this can be rewritten as
\begin{equation}
\label{eq:tilde.w}
t \dt (\theta w) + A\left(t\kappa \vf'(u) \theta w\right)+Q   \theta w =\left (\theta\,t \dt\Big( t \frac{\dt \kappa}{\kappa}\Big) +\tilde{\rho} Q\right)\frac{\vf(u)}{\vf'(u)}.
\end{equation}
It follows easily from hypotheses~\eqref{m} and~\eqref{eq:conditions.kappa} that if we take $\rho$ large enough then $\tilde\rho$ is positive and large enough so that
\[
Q> 0,\qquad
\theta\,t\dt\Big( t \frac{\dt \kappa}{\kappa}\Big) +\tilde{\rho} Q > 0.
\]
Thus, if we multiply equation~\eqref{eq:tilde.w} by $\imath = -\mathds{1}_{\{\theta w\leq 0\}}$, we get that 
\ben
 t\dt \int  (\theta w)^- + B + Q \int (\theta w)^- \le 0,
\enn
where $B= \int \imath A\left(t\kappa \vf'(u) \theta w\right)$.  Since $\imath$ is a non-decreasing function of $\kappa \vf'(u) (\theta w)$, property~\eqref{A3} implies $B\ge0$. Hence $\dt \int (\theta w)^- \leq 0$. On the other hand,  $\theta w (0)\ge0$. Therefore, since $(\theta w)^-$ is non-negative, it is identically 0 for $t\ge 0$, and hence $\theta w\ge0$.

To conclude, we notice that 
\ben
\partial_t(\theta t^{\rho\theta}\kappa(t,x)\varphi(u(t,x))) =t^{\rho\theta-1}\kappa \varphi'(u)\theta w\ge0,
\enn
which implies~\eqref{eq:monotonicity.separable}. 
\end{proof}


\subsection{The general (non-separable) case}
The monotonicity formula~\eqref{eq:monotonicity.separable} can be extended to equations in the general non-separable form~\eqref{eq:non.separable}
\begin{theo}\label{theo:non.separable}
	Let $u$ be a non-negative classical solution to \eqref{eq:non.separable} on $[0,T]$, belonging to $L^1(\Rd)\cap \Linf(\Rd)$.  Assume that $F(t,x,u)$ is non-decreasing in~$u$, satisfies $F(t,x,0)=0$, 
	$$
	\frac{F_t}{F}, \Big(\frac{F_t}{F}\Big)_t, \frac{F_{u,t}}{F_u}
	\quad\text{are bounded},
	$$ 
	and for some $m>0$, $\theta\in \{-1,1\}$
	\beq\label{mforF}
	\inf_{u\geq 0,\, x\in\mathbb{R}^d,\, t\in[0,T]}\left\{ \theta \frac {F_{uu}F} {F_u^2}\right\} \geq m.
	\enq
	Then,  there exists a constant $\rho>0$, independent of $u$, such that 
	\beq
	\label{eq:monotonicity.non-separable}
	t \to \theta t^{\rho\theta}F(t,x,u(t,x)) \quad \text{is non-decreasing in }[0,T].
	\enq
	
\end{theo}

\begin{proof} 
Let 	$w:=t\dt u + (\theta\rho + t\frac{F_t}{F})\frac{F}{F_u}$ with $\rho>0$ to be fixed later.
Differentiating~\eqref{eq:non.separable} we now have
\ben
\dtt u +\frac{1}{t}A(F_u(t\frac{F_t}{F_u}+ t\dt u))=0,
\enn 
or equivalently
\ben
\dtt u +\frac1t A\left(F_u\Big(w - \theta\rho\frac{F}{F_u}\Big)\right)=0,
\enn 
while
\ben
\dt w &=& t \dtt u + \dt u + \left(\Big(\theta\rho+t\frac{F_t}{F}\Big)\Big(1-\frac{F F_{uu}}{F_u^2}\Big) + t\Big(\frac{F_t}{F}\Big)_u \frac{F}{F_u} \right)\dt u \\
&&+ \Big(\theta\rho+t\frac{F_t}{F}\Big)\Big(\frac{F}{F_u}\Big)_t  + \Big(t\frac{F_t}{F}\Big)_t\frac{F}{F_u}.
 \enn
Combining these equations, we arrive to 
\ben
\dt w + A(F_uw) &=& \dt u\left( 1+ t\frac{F_t}{F}+ t\Big(\frac{F_t}{F}\Big)_u\frac{F}{F_u} - \tilde \rho\theta \frac{F_{uu}F}{F_u^2}  \right)\\[10pt]
&& +\theta \tilde \rho \Big(\frac{F}{F_u}\Big)_t + \Big(t\frac{F_t}{F}\Big)_t \frac{F}{F_u}\\[10pt]
&=& -\frac1t\Big(w - \theta\tilde\rho\frac{F}{F_u}\Big) Q
+ \theta\tilde \rho \Big(\frac{F}{F_u}\Big)_t + \Big(t\frac{F_t}{F}\Big)_t \frac{F}{F_u},
\enn
where 
\ben 
\tilde \rho=\rho+\theta\, t\frac{F_t}{F},\qquad Q = -\left( 1+ t\frac{F_t}{F} + t\Big(\frac{F_t}{F}\Big)_u\frac{F}{F_u} \right) + \tilde \rho \theta \frac{F_{uu}F}{F_u^2}.
\enn
Then
\ben
t\dt (\theta w) + A(tF_u\theta w) + Q\theta w =   \left(\tilde \rho Q + t\Big(\frac{F}{F_u}\Big)_t\frac{F_u}{F} + 
\theta\,t\Big(t\frac{F_t}{F}\Big)_t \right)\frac{F}{F_u}.
\enn
It follows easily from the assumptions on $F$ that if we take $\rho$ large enough, then $\tilde\rho$ is positive and large enough so that
$$
Q>0,\qquad \tilde \rho Q + t\Big(\frac{F}{F_u}\Big)_t\frac{F_u}{F} + 
\theta\,t\Big(t\frac{F_t}{F}\Big)_t >0,
$$
and  the result follows as in the proof of Theorem~\ref{theo-u}. 

Note that 
\[
\Big(\frac{F}{F_u}\Big)_t\frac{F_u}{F}= \frac{F_tF_u-F F_{u,t}}{F_uF}
= \frac{F_t}{F} - \frac{F_{u,t}}{F_u}.
\] 
\end{proof} 
\subsection{Consequences for fully nonlinear parabolic equations}

We discuss here implications for the models studied in  \cite{AbergelLoeper, BoLoZo2}. 

We assume that $v$ is a classical solution to \eqref{eq:intro-v} on $[0,T]$, and hence that $u=-A(v)$ solves equation \eqref{eq:non.separable}.  We start by proving an auxiliary result.
\begin{lemme}
Let $v$ be a locally bounded classical solution  to \eqref{eq:intro-v} on $[0,T]$, with $F$ satisfying the assumptions of~Theorem~\ref{theo:non.separable} with $\theta=1$, and with initial data $v_0$.   Given $M>0$, there exists $M'> M$ such that 
\ben
	-A(v_0) \geq M'\implies -A(v) \geq M\quad  \text{for }t\in[0,T].
\enn
\end{lemme}
\begin{proof} Let $b(t,x)=F\big(t,x,-A(v(x,t))\big)$. 
Since, by assumption, $F_t\ge -\ell F$, then 
	\ben
	\dt b = -F_u\, A(b) + F_t\geq -F_u\, A(b) - \ell b.
	\enn
Therefore, $g(t,x)=b(t,x) e^{\ell t}$ satisfies $\dt g \geq -F_u\, A (g)$, while any constant $k$ satisfies $\dt k = - F_u \, A(k)$. 	Multiplying $\partial_t(k-g)$ by $\mathds{1}_{g \leq k}$ and using property~\eqref{A3}, we conclude that $g$ remains larger than $k$ if it was so at the initial time.  Take $M'$ large so that $b(0,x)= F(t,x,-A(v_0(x,t)))$ is larger than $k>0$ to be determined. Since $F(t,x,-A(v(x,t)))\ge e^{-\ell T}k$, then $F(t,x,-A(v(x,t)))$ is large if $k$ is large enough, and we conclude that $-A(v)$ can be made as large as desired. 
\end{proof}

We now consider $\bar{v}$ solution to \eqref{eq:intro-v} such that 
\ben
A(\bar{v}_0)&=&\min\{A(v_0),-M' \},
\enn
with $M'$ as above. 
Then,  $\bu=-A(\bv)$ is a solution to \eqref{eq:non.separable} and, by the comparison principle
\ben
-A(v) \leq -A(\bv)\quad \text{holds for all time }t\geq 0.
\enn
Now, thanks to the monotonicity formula~\eqref{eq:monotonicity.non-separable}, we will prove the interior reguarity of $v$. 
\begin{theo}\label{theo-v}
Let $v\in L^\infty_{loc}([0,T)\times \R)$ be a classical solution to~\eqref{eq:intro-v}
with $F$ satisfying \eqref{mforF} with $\theta=1$ on $[M,+\infty)$ for some $M>0$, and the rest of the conditions of Theorem~\ref{theo:non.separable}.
If $\bv\in L^\infty_{loc}([0,T)\times \R)$, then 
\[
F^+(t,x,-A(v)) \in L^\infty_{loc}((0,T)\times \R),
\] 
with bounds  that depend only on $\bv_0$, $\bv$, $m$, and $t$.

Assuming moreover that either $A(v)$ is bounded from above or that $F_u(t,x,u)$ is bounded away from $0$ and $+\infty$  for $u<0$, then $v(t,\cdot)\in C^{2,\alpha}$ uniformly on $[\tau,T]$ for $\tau>0$.
\end{theo}

\begin{proof} If $\bv$ is locally bounded,  it follows from the auxiliary lemma that   $F(t,x,\cdot)$ satisfies the assumptions of Theorem~\ref{theo:non.separable} at $-A(\bar v)$ for all $t\in[0,T]$.
Therefore, Theorem~\ref{theo:non.separable} applies.
Using thethe monotonicity formula~\eqref{eq:monotonicity.non-separable} with $\theta=1$ for $0 < t_1 \leq t_2 \leq T$,
\ben
\bv(t_2,x) &=& \bv(t_1,x) + \int_{t_1}^{t_2} F(s,x,-A(\bv))ds\\
&\geq& \bv(t_1,x) + F(t_1,x,-A(\bv(t_1)))\int_{t_1}^{t_2} \frac{t_1^\rho}{s^\rho}ds\\
&\geq& \bv(0,x) + F(t_1,x,-A(\bv(t_1)))\int_{t_1}^{t_2} \frac{t_1^\rho}{s^\rho}ds,
\enn
which yields the stated boundedness of $F(t_1,x,-A(\bv))$.

The second point follows from the first, as, now, $F$ is uniformly elliptic, and standard theory applies.
\end{proof}

\section{Expansion of positiviy and application to the degenerate case}\label{sec-degen} 
\label{sect-degenerate.case} \setcounter{equation}{0}

We consider the case 
\[
\dt v = F(t,x,D^2v)=\bar{F}(t,x,A_{ij}(M-D^2v)^{ij}),\label{eq.final}
\]
where $A$ and $M$ are symmetric positive matrices and  $(M-D^2v)^{ij}$ is  the inverse of $M-D^2v$. By elementary affine transformations one can assume $A=M=I$ the identity matrix. We also assume that
$\bar{F}=\bar{F}(t,x,z)$ satisfies 
\beq\label{assum-m}
\bar{F}_z(t,x,z) \sim |z|^{m-1} \quad\text{as } z\sim0
\enq
for some $m\in[0,1]$,
that $\bar{F}$ is smooth with respect to the other variables, and
\beq
{F}(t,x,D^2v) \leq C\quad\text{on }B_r(0)\label{assum2}.
\enq
We further assume that 
\beq\label{assum3}
F_x \in \Linf_{loc}((0,T)\times \mathbb{R}^d;\Linf (\mathbb{R}_+));
\enq
that is, for compact sets $K \subset (0,T)\times \mathbb{R}^d$, $F_x \in \Linf(K\times\mathbb{R}_+)$.

The problem is defined for $(I-D^2)v$ non-negative. Hence, $\upsilon=|x|^2/2-v$ is convex,  and we can consider its lower semi-continuous Legendre transform 
\ben
\upsilon^*(y) = \sup_{x}\big( x\cdot y -\upsilon(x)\big).
\enn
When $\upsilon$ is lower semi-continuous and its supremum is attained at a point~$(x,y)$ where $\upsilon$ is twice differentiable, then
\ben
x =D \upsilon^*(y), \qquad D^2\upsilon^*(y)=[D^2\upsilon(x)]^{-1}.
\enn
Moreover, if $\upsilon$ depends smoothly on $t$, 
\ben
\dt \upsilon(t,x) + \dt \upsilon^*(t,y) = 0.
\enn

The equation satisfied by $\upsilon^*$ is now 
\beq
\dt \upsilon^* = \bar{F}(t,D_y \upsilon^*,\Delta \upsilon^*).\label{eng0}
\enq
Note that~\eqref{assum2} implies that  $0 \leq D^2\upsilon^* \leq C$ on $B_r(0)$.  
Here we establish an independent result for this parabolic equation, on the condition that the solution is convex. 
\begin{theo}\label{theo:degen}
Let $\bar{F}\in C^1_{loc}([0,T]\times \mathbb{R}^d \times \mathbb{R})$ having the behaviour~\eqref{assum-m} for some $m\in [0,1]$. Assume that $\upsilon^*$ is a convex solution to \eqref{eng0} such that $\Delta \upsilon^*$ is bounded from above, and not identically 0 until time $T$. Then for $t>0$, $\upsilon^*$ is $C^2$ smooth in $y$  and $\Delta \upsilon^*$ is bounded away from $0$ locally uniformly on $(0,T)\times \mathbb{R}^d$.
\end{theo}

\begin{proof} If $m=1$ the problem is uniformly elliptic, and the result is well known, so we  assume $m \neq 1$. 
	
Let $u=\Delta \upsilon^*$. Then, 
\beq
\dt u =\text{div}(\bar{F}_xD^2 \upsilon^*)+ \text{div}(\bar{F}_z \nabla u).\label{eng2}
\enq
The proof is done by Moser iterations. We follow the technique of \cite{GT} that we adapt from  the elliptic to the parabolic case. We first observe that from the convexity of $\upsilon^*$ and the fact that $\Delta \upsilon^*$ is bounded,
$D^2 \upsilon^*$ is bounded,  and $D_{ij} \upsilon^* \leq u$.
Multiplying \eqref{eng2} by $\eta^2(y)u^\beta$ for $\beta<0, \beta \neq -m,-1$ we obtain
\begin{equation}
\label{moser1}
	\begin{array}{l}
\displaystyle|\beta| \int_{t_1}^{t_2}\int_{\Rd} (\eta \dy(u^{\frac{\beta+m}{2}})\frac{2}{\beta+m})^2 \leq
C\Big(\frac{1}{\beta+1}\int_{\Rd} \eta^2 u^{\beta+1}(t,y) \,dy \Big|^{t_2}_{t_1} \\[10pt]
\displaystyle\qquad+ \int_{t_1}^{t_2}\int_{\Rd}(\dy \eta)^2 (u^{\beta+m}+u^{\beta+1}) + \eta^2 u^{\beta+2-m}\Big),
\end{array}
\end{equation}
where $C$ depends on our assumptions on $\bar{F}$ and the bound on $u$.
If $\beta=-m$ we obtain :
\[
\begin{array}{l}
\displaystyle m\int_{t_1}^{t_2}\int_{\Rd} (\eta \dy(\ln u))^2 \leq
 C\Big(\frac{1}{1-m}\int_{\Rd} \eta^2 u^{1-m}(t,y)\, dy \Big|^{t_2}_{t_1} \\[10pt]  
\displaystyle\quad+ \int_{t_1}^{t_2}\int_{\Rd}(\dy \eta)^2 (1 + u^{1-m}) + \eta^2 u^{2-2m}\Big).
\end{array}
\]
Following \cite[Section 8.6]{GT}  the second bound yields that 
\ben
\fint_{[t_1,t_2]}  \int_{B_r(0)} |\dy(\ln u)| dy \leq Cr^{n-1},
\enn
 and hence by \cite[Theorem 7.21]{GT}  that for some $p_0>0$ and $l=\frac{1}{|B_r|}\int_{B_r} \ln u$ there holds
\ben
\fint_{[t_1,t_2]}\int_{B_r} e^{p_0|\ln u - l |} dy \leq D.
\enn
Note that $C,D$ here might depend on $\|u\|_{L^\infty([t_1,t_2]\times B_r)}$ which we control anyway.
This in turn implies 
\ben
\left(\fint_{[t_1,t_2]}\int_{B_r} u^{p_0}\right)\left( \fint_{[t_1,t_2]}\int_{B_r} u^{-p_0}\right) \leq D,
\enn
which gives a bound on $\fint_{[t_1,t_2]}\int_{B_r} u^{-p_0}$ depending also on $\big(\fint_{[t_1,t_2]}\int_{B_r} u^{p_0}\big)^{-1}$. 

From \eqref{moser1} using the boundedness of $u$ and fixing some $\theta \in (0,1)$  we deduce
\[
\begin{array}{l}
\displaystyle
\int_{t_1}^{t_2}\int_{B_{\theta r}}(\dy(u^{\frac{\beta+m}{2}})\frac{2}{\beta+m})^2\, dtdy\leq  \\[10pt]
\displaystyle\qquad
\frac{C}{|\beta|}\frac{1}{|\beta+1|}\int_{B_r} u^{\beta+1}(t,y)\, dy \Big|^{t_2}_{t_1}
+ r^{-2}\int_{t_1}^{t_2}\int_{B_r}u^{\beta+m}\,dtdy.
\end{array}
\]
Sobolev's inequality will then yield a control on $\|u\|_{q(m)}$, for 
\[q(m)=\frac{\beta+m}{2}\frac{2d}{d-2}=\frac{(\beta+m)d}{d-2}\]
if $d \geq 2$ and $+\infty$ otherwise.
By starting with $\beta+m=-p_0$ above, and classically iterating Sobolev's injection this gives a bound of the form
\ben
\sup_{y\in B_r} \frac{1}{u(t_3,y)} \leq C(r, m,  \|u\|_{L^\infty([t_1,t_2]\times B_r}, \frac{1}{\inf_{[t_1,t_2]}\{\|u\|_{L^{p_0}(B_r)}\}}, \nu)
\enn 
for $t_3\in [t_1+\nu, t_2-\nu]$.
 Equation \eqref{eng2} becomes now uniformly elliptic, and we obtain that $u\in C^\alpha$. As $u=\Delta \upsilon^*$, classical elliptic regularity then yields $\upsilon^* \in C^{2,\alpha}_y$.
 \end{proof}

\noindent\emph{Remarks. }  (i) When $d>1$, this  theorem does not imply that $D^2 \upsilon$ is uniformly positive.

\noindent (ii) Equation \eqref{eng2} and our result is somehow similar to the porous medium like equation addressed in~\cite{DiBenedettoDegen}; see equation 5.1 of Chapter 3, and the proof in Proposition 7.2 of Chapter 4 about expansion of positivity for singular porous medium equations. However in our present case the a priori knowledge that $D^2\upsilon^*$ is positive and bounded considerably simplifies the estimates. 

\noindent (iii)  The presence of the term $\frac{1}{\inf_{[t_1,t_2]}\{\|u\|_{L^{p_0}(B_r)}\}}$ in the estimate implies that it is valid up to extinction. Indeed, before extinction, there exists always $R$ large enough so that $\|u\|_{L^{p_0}(B_r)}$ is bounded away from 0. Extinction in our case means that $\Delta \upsilon^* \equiv 0$, hence that $\Delta v \equiv -\infty$ which does not occur if there is a bounded subsolution to \eqref{eq:intro-v}.

\noindent (iv) If we remain in a class of solutions to \eqref{main} in which the comparison principle holds, then the {\it expansion of positivity} result of Theorem~\ref{theo:degen} should remain valid  without assuming that $\Delta\upsilon^*$ is bounded from above. Equivalently, one can write that $\min\{v,\frac{C}{\vf^{-1}(\kappa)}\}$ is a supersolution to \eqref{main} and proceed with the estimates.
 
\medskip

As a corollary, we have an interior lower bound for Laplacian of solutions to~\eqref{eq:intro-v}. 
\begin{theo}\label{theo-degen2} Let $v$ be a solution to~\eqref{eq:intro-v}. Assume that $F$, $\bar{F}$ and $v$ satisfy \eqref{assum-m}--\eqref{assum3}. Then $\Delta v$ admits an interior lower bound in $B_{\theta r}(0)$ for $\theta<1$.
\end{theo}
\begin{proof} 
Theorem~\ref{theo:degen} implies that $\upsilon$ is bounded away from $+\infty$, and hence that the $D^2v$ as a matrix is bounded from below (i.e. its eigenvalues are bounded away from $-\infty$).
\end{proof}


\section{Consequence for fully non-linear Hamilton-Jacobi-Bellman equations}
\label{sect-consequence.fnl} \setcounter{equation}{0}

This section is motivated by the papers \cite{AbergelLoeper, LoMi1, BoLoSonZou} of the first author, where fully non-linear versions of the Black-Scholes equation are considered in the context of financial derivatives pricing with {\it market impact}. 
We are in dimension $d=1$,  $A=-\dxx$, and $F(t,x,\gamma):([0,T]\times \R \times \R) \to \R$ satisfies the assumptions of Theorem~\ref{theo:non.separable}, for  $\gamma >0$ and such that 
$F_\gamma \sim \gamma^{m-1}$ for $
\gamma<0$ with $m\in [0,1]$.

Considering again equation \eqref{eq:intro-v}, but backwards in time (as is usually the case for stochastic control problems)
\beq\label{eq:HJB}
\dt v +F(t,x,\dxx v)=0,
\enq
for which, we assume that the classical solution $u$ is locally bounded.
By combining Theorems \ref{theo:non.separable} and \ref{theo:degen} we obtain the following interior regularity result.

\begin{theo}\label{theo:reg-aber-loep}
Under the above assumptions,  the solution to \eqref{eq:HJB}
belongs to $C^{2,\alpha}(\R)$ for $0\leq t<T-\tau$ for any $\tau>0$.
In particular, the result applies to the solution of \eqref{eq.aberloep} if $0<p_1 \leq 1$, $p_1 \leq p_2$, $\kappa$ satisfies the assumptions of Theorem~\ref{theo-u}, and  $\dx \kappa$ is bounded.  
\end{theo}

This bound also has probabilistic interpretation:
We consider the associated stochastic differential equation 
\[
dX_t=\sigma(t,X_t)dW_t,\qquad
\sigma^2(t,X_t)= 2 \kappa(t,X_t)\varphi'(\dxx v(t,X_t)),
\]
which corresponds to the linearized equation.
As done in \cite{AbergelLoeper, LoMi1, BoLoSonZou}, we have
\ben
\dt (\kappa(t,x)\vf(\dxx v)) + \kappa \vf'(\dxx v) \dxx(\kappa \vf(\dxx v)= \frac{\dt \kappa}{\kappa}(\kappa(t,x)\vf(\dxx v)).
\enn
We thus have (under  assumptions  that guarantee that the representation formula holds) that for $V_t=\kappa\varphi(\dxx v)(t,X_t)$,
\ben
V(t,x) = \E_{t,x}\left(V(T,X^{t,x}_T) e^{-\int_t^T \dt \kappa/\kappa}\right).
\enn
The interior bound on $\vf(\dxx v)$  implies that the stochastic differential equation is well defined on $[0,T)$, and that
\ben
\bP(\vf(\dxx v(T,X_T))=+\infty) = 0.
\enn


\

\noindent{\large \textbf{Acknowledgments}}

\noindent We thank J.L. V\'azquez for useful comments.

\noindent FQ was supported by projects MTM2014-53037-P and MTM2017-87596-P (Spain).

\bibliography{biblio-greg}

\def\cprime{$'$} \def\cprime{$'$}
\begin{thebibliography}{10}

\bibitem{AbergelLoeper}
Fr\'ed\'eric Abergel and Gr\'egoire Loeper.
\newblock Pricing and hedging contingent claims with liquidity costs and market
  impact.
\newblock To appear in the proceedings of the International Workshop on
  Econophysics and Sociophysics, Springer, New Economic Window, 2016.

\bibitem{Aronson-Benilan}
Donald~G. Aronson and Philippe B\'enilan.
\newblock R\'egularit\'e des solutions de l'\'equation des milieux poreux dans
  {${\bf R}^{N}$}.
\newblock {\em C. R. Acad. Sci. Paris S\'er. A-B}, 288(2):A103--A105, 1979.

\bibitem{BoLoSonZou}
Bruno Bouchard, Gr{\'e}goire Loeper, Halil~Mete Soner, and Chao Zhou.
\newblock Second order stochastic target problems with generalized market
  impact.
\newblock {\em arXiv preprint arXiv:1806.08533}, 2018.

\bibitem{BoLoZo1}
Bruno Bouchard, Gr\'{e}goire Loeper, and Yiyi Zou.
\newblock Almost-sure hedging with permanent price impact.
\newblock {\em Finance Stoch.}, 20(3):741--771, 2016.

\bibitem{BoLoZo2}
Bruno Bouchard, Gr\'{e}goire Loeper, and Yiyi Zou.
\newblock Hedging of covered options with linear market impact and gamma
  constraint.
\newblock {\em SIAM J. Control Optim.}, 55(5):3319--3348, 2017.

\bibitem{CranPieA}
Michael Crandall and Michel Pierre.
\newblock Regularizing effects for {$u_{t}+A\varphi (u)=0$} in {$L^{1}$}.
\newblock {\em J. Funct. Anal.}, 45(2):194--212, 1982.

\bibitem{CranPieDelta}
Michael~G. Crandall and Michel Pierre.
\newblock Regularizing effects for {$u_{t}=\Delta \varphi (u)$}.
\newblock {\em Trans. Amer. Math. Soc.}, 274(1):159--168, 1982.

\bibitem{DiBenedettoDegen}
Emmanuele DiBenedetto, Ugo Gianazza, and Vincenzo Vespri.
\newblock {\em Harnack's inequality for degenerate and singular parabolic
  equations}.
\newblock Springer Monographs in Mathematics. Springer, New York, 2012.

\bibitem{GT}
D.~Gilbarg and N.~S. Trudinger.
\newblock {\em Elliptic partial differential equations of second order}, volume
  224 of {\em Grundlehren der Mathematischen Wissenschaften [Fundamental
  Principles of Mathematical Sciences]}.
\newblock Springer-Verlag, Berlin, second edition, 1983.

\bibitem{GuoLo}
Ivan Guo and Gregoire Loeper.
\newblock Path dependent optimal transport and model calibration on exotic
  derivatives.
\newblock {\em arXiv preprint arXiv:1812.03526}, 2018.

\bibitem{GuoLeoLo}
Ivan Guo, Gr{\'e}goire Loeper, and Shiyi Wang.
\newblock Local volatility calibration by optimal transport.
\newblock {\em arXiv preprint arXiv:1709.08075}, 2017.

\bibitem{LoMi1}
Gregoire Loeper.
\newblock Option pricing with linear market impact and nonlinear
  {B}lack-{S}choles equations.
\newblock {\em Ann. Appl. Probab.}, 28(5):2664--2726, 2018.

\bibitem{TanTouz}
Xiaolu Tan and Nizar Touzi.
\newblock Optimal transportation under controlled stochastic dynamics.
\newblock {\em Ann. Probab.}, 41(5):3201--3240, 2013.

\end{thebibliography}

\

\noindent\textbf{Addresses:}

\noindent\textsc{Gregoire Loeper: }
\\ Monash University 
\\ School of Mathematics,
\\ 9 Rainforest Walk 
\\ 3800 Clayton Vic, Australia
\smallskip
\\ email: gregoire.loeper@monash.edu

\bigskip

\noindent\textsc{Fernando Quir\'{o}s: } 
\\ Departamento de Matem\'{a}ticas 
\\ Universidad Aut\'{o}noma de Madrid
\\ 28049 Madrid, Spain
\smallskip
\\ e-mail: fernando.quiros@uam.es

\end{document}